\documentclass[11pt, notitlepage]{article}
\usepackage{amssymb,amsmath,comment}
\catcode`\@=11 \@addtoreset{equation}{section}
\def\thesection{\arabic{section}}

\def\theequation{\thesection.\arabic{equation}}
\catcode`\@=12
\usepackage{caption}
\usepackage{subcaption}
\usepackage{colortbl}%
\usepackage{a4wide}
\usepackage{parskip}
\usepackage{hyperref}
\usepackage{cite}

\newcommand{\ds} {\displaystyle}
\newcommand{\e}{\epsilon}
\newcommand{\pa} {\partial}
\newcommand{\al} {\alpha}
\newcommand{\ba} {\beta}

\newcommand{\de} {\delta}

\newcommand{\Om} {\Omega}
\newcommand{\ra} {\rightarrow}

\newcommand{\De} {\Delta}
\newcommand{\la} {\lambda}

\newcommand{\noi} {\noindent}
\newcommand{\na} {\nabla}

\newcommand{\mb} {\mathbb}
\newcommand{\mc} {\mathcal}

\newcommand{\ld} {\langle}
\newcommand{\rd} {\rangle}
\newcommand{\I}{\int\limits_}

\usepackage[all]{xy}
\catcode`\@=11
\def\theequation{\@arabic{\c@section}.\@arabic{\c@equation}}
\catcode`\@=12

\def\QED{\hfill {$\square$}\goodbreak \medskip}

\newtheorem{Theorem}{Theorem}[section]
\newtheorem{Lemma}[Theorem]{Lemma}
\newtheorem{Proposition}[Theorem]{Proposition}

\newtheorem{Remark}[Theorem]{Remark}
\newtheorem{Definition}[Theorem]{Definition}

\def\XXint#1#2#3{{\setbox0=\hbox{$#1{#2#3}{\int}$ }
		\vcenter{\hbox{$#2#3$ }}\kern-.6\wd0}}

\begin{document}
	{\vspace{0.01in}
		\title
		{A Choquard type equation with a singular absorption nonlinearity in two dimension}

		\author{ {\bf G.C. Anthal\footnote{	Department of Mathematics, Indian Institute of Technology, Delhi, Hauz Khas, New Delhi-110016, India. e-mail: Gurdevanthal92@gmail.com}\;, J. Giacomoni\footnote{ 	LMAP (UMR E2S UPPA CNRS 5142) Bat. IPRA, Avenue de l'Universit\'{e}, 64013 Pau, France. 
		email: jacques.giacomoni@univ-pau.fr}  \;and  K. Sreenadh\footnote{Department of Mathematics, Indian Institute of Technology, Delhi, Hauz Khas, New Delhi-110016, India.	e-mail: sreenadh@maths.iitd.ac.in}}}
		\date{}
		\maketitle
		
		\begin{abstract}
			\noi In this article, we show the existence of a nonnegative solution to the singular problem $(\mc P_\la)$ posed in a bounded domain $\Omega$ in $\mb R^2$ (see below). We achieve this by approximating the singular function $u^{-\beta}\log(u)$ by a function $l_\e(u)$ which pointwisely converges to $-u^\beta\log(u)$ as $\e \ra 0$. Using variational techniques, the perturbed equation  $-\De u+l_\e(u)=\ds\la \left(\I{\Om}\frac{F(u(y))}{|x-y|^\mu}dy\right)f(u(x))$ is shown to have a solution $u_\e \in H_0^{1}(\Om)$ when the parameter $\la >0$ is small enough. Letting $\e \ra 0$ and proving a pointwise gradient estimate, we show that the solution $u_\e$ converges to a nontrivial nonnegative solution of the original problem $(\mc P_\la)$.
			\medskip
			
	\noi \textbf{ Keywords}: Absorption singular nonlinearity, non local problem, Hardy-Littlewood-Sobolev inequality, variational methods, exponential growth. \\
	
	\noi \textit{2020 Mathematics Subject Classification:} 35A15, 35B20, 35D30, 35J20, 35J60, 35J75.
			

		\end{abstract}
		\section {Introduction}
		This article is concerned with the study of following problem involving both an absorption type singularity and exponential nonlinearity of Choquard type
		\begin{equation*}
			(\mc P_\la)  \left\{	-\De u =u^{-\ba} \log(u)\chi_{\{u>0\}} +\la \left(\I{\Om}\frac{F(u(y))}{|x-y|^\mu}dy\right) f(u),~\text{in}~\Om,~u \geq 0,~\text{in}~\Om,~u=0~\text{on}~\pa \Om,	\right.
		\end{equation*}
		where $\Om$ is a smooth bounded domain in $\mb R^2$ and $0<\mu<1$. 
		Here $\chi_{\{u>0\}}$ denotes the characteristic function corresponding to the set $\{x \in \Om:u(x)>0\}$ and by convention $u^{-\ba}\log u  \chi_{\{u>0\}}=0$ if $u\leq0$.\\
			The notion of weak solution is as below:
		\begin{Definition}
		By a weak solution of problem $(\mc P_\la)$, we mean a function $u \in H_0^{1}(\Om)$ such that 
		\begin{equation*}
		u^{-\beta}	\log u  \chi_{\{u>0\}} \in L^1_{loc}(\Om)
		\end{equation*}
		and \begin{equation*}
			\I{\Om}\na u\na \varphi=\I{\Om \cap \{u>0\}} u^{-\ba}\log (u) \varphi +\I{\Om}\I{\Om}\frac{F(u(y))f(u(x))\varphi(x)}{|x-y|^\mu},~\text{for any}~\varphi \in C_c^1(\Om).
		\end{equation*}

	\end{Definition}
		Here $C_c^1(\Om)$ stands for the functions belonging to $C^1(\Om)$ with compact support in $\Omega$. 
		We assume the following hypothesis on the nonlinearity $f$:
		\begin{itemize}
			\item [$(f_1)$] $f$ is of class $C^{1,\beta}(0,\infty) \cap C[0,\infty)$ for some $0<\beta<1$.
			\item [$(f_2)$] For all $\al>0$, $$\lim\limits_{t \ra \infty}\frac{|f(t)|}{\exp(\al t^2)}=0.$$
			\item [$(f_3)$] There exists constant $1<r_0<2$ such that
			\begin{equation*}
 				\limsup\limits_{t \ra 0^+} \frac{|f(t)|}{t^{r_0}}< \infty.
			\end{equation*} 
			\item [$(f_4)$]There exists $l >0$ such that $\ds t \ra \ds \frac{f(t)}{t^l}$ is increasing on $\mb R^+ \setminus \{0\}$.
			\item[$(f_5)$] There exist $T$, $T_0 >0$ and $\gamma_0 >0$ such that $0<t^{\gamma_0}F(t)\leq T_0f(t)$ for all $|t|>T$.
			\item [$(f_6)$] $\limsup\limits_{t\ra 0^+} |f^{\prime}(t)| < \infty$.
		\end{itemize}
	\begin{Remark}
		Using 
		$(f_5)$ it is easy to show that there are constants $A$, $t_0\geq 0$ and $\gamma>1$ such that 
		\begin{equation}\label{1.3}
		F(t)\geq At^{\gamma} ~\text{for all}~ t\geq t_0.
		\end{equation}
Furthermore using $(f_3)$, we see that $f(0)=0$ and $f'(0)=0$.
Hence we can extend $f$ by $0$ on $\mb R^-$ and then as a $C^1$ function on $\mb R$.
\end{Remark}

We state our main result:
\begin{Theorem}\label{tmr}
Suppose that $f$ satisfies $(f_1)-(f_6)$. Then there exists a $\bar{\la}>0$ such that for each $0<\la<\bar{\la}$, problem $(\mc P_\la)$ has a nonnegative nontrivial solution $u$.
\end{Theorem}
We employ a variational approach to solve $(\mc P_\la)$. More precisely, we first study the following approximated problem:
\begin{equation*}
(\mc P_{\e,\la}) \left\{	-\De u + l_\e(u) =\la \left(\I{\Om}\frac{F(u(y))}{|x-y|^\mu}dy\right) f(u(x)),~\text{in}~\Om,~u \geq 0,~\text{in}~\Om,~u=0~\text{on}~\pa \Om,\right.
\end{equation*}
where \begin{equation*}
l_{\epsilon}(t) = \begin{cases}
	\ds-\frac{t^q}{(t+\e)^{\ba+q}}\log \left(t +\frac{\epsilon}{t +\epsilon}\right),&~t\geq 0,\\
	0,&t <0,
\end{cases}
\end{equation*}
and $0<q<1$ is such that $q <r_0-1$. The energy functional associated to the  problem $({\mc P_{\epsilon,\la}})$ is given by
\begin{equation}\label{eq2.1}
J_{\epsilon,\la}(u)=\frac{\|u\|^2}{2}+\int\limits_{\Om} L_{\epsilon}dx-\frac{\la}{2}\I{\Om}\I{\Om}\frac{F(u(y))}{|x-y|^\mu}F(u(x))dydx,~\forall\,u \in H_0^{1}(\Om),
\end{equation}
where $\ds L_{\epsilon}(t) = \int\limits_{0}^{t}l_\epsilon(s)ds$ and $\ds F(t)= \int\limits_{0}^{t}f(s)ds$. In this matter we shall discuss about estimates of $J_{\e,\la}$ independently of $\e$ in Section \ref{P}. From the geometry of $J_{\e,\la}$ and using variational arguments, we prove the existence of a nontrivial solution, $u_\e$,  to $(\mc P_{\e,\la})$. Next from a priori estimates and a key point pointwise gradient estimate (in the spirit of works \cite{FMS, LM, S}, see Section 4) independent of $\e$, we show that $u_\e$ converges as $\epsilon\to 0^+$ to a nontrivial solution of $(\mc P_\la)$.

For reader's convenience, we now present a brief introduction to existence and multiplicity results for the equations involving singular and Choquard nonlinearities.
The equations involving Choquard nonlinearity were introduced by Ph. Choquard in 1976 in the modeling of one-component plasma. Choquard equations are also used to study the model of the polaron, where free electrons in an ionic lattice interact with photons associated to deformations of the lattice or with the polarisation that it creates on the medium. For further details, see for instance \cite {HF,Li,Pe}. These applications lead to a vast investigation of nonlocal equations involving Choquard nonlinearity. Without any attempt to provide the complete list, we refer to  \cite{GY,MS4} and references therein for a large overview of current results concerning Choquard problems.\\
In 1996, Joao Marcos Do \'{O} in \cite{M} studied the following problem
\begin{equation*}
\left\{	-\De_n u = f(x,u) ~\text{in}~\Om, ~u\geq 0 ~\text{in}~\Om,\right.
\end{equation*}
where $\Om$ is a smooth bounded domain in $ \mb R^n$ with $n \geq 2$ and the growth of the nonlinearity $f(x,u)$ is as $\exp(\al |u|^{n/(n-1)})$ when $|u| \ra \infty$. The author showed the existence of a nontrivial solution under suitable technical assumptions on the nonlinearity $f$. We also refer to \cite{FOZ, FMR, R} for related results on critical elliptic equations and systems in two dimension. R. Arora and et. al. \cite{AGMS} extended the study to the following Kirchoff equation with exponential nonlinearity of Choquard type 
\begin{equation}\label{r}
\left\{ -m\left(\I{\Om}|\na u|^ndx\right)\De_n u=\left(\I{\Om}\frac{F(y,u)dy}{|x-y|^\mu}f(x,u)dx\right),~u >0~\text{in}~\Om,~u=0~\text{on}~\pa \Om,\right.
\end{equation}
where $\mu \in  (0,n)$, $\Om$ is a smooth bounded domain in $\mb R^n$, $n \geq 2$, $m: \mb R^+\ra \mb R^+$ and $f:\Om \times \mb R \ra \mb R$ are continuous functions satisfying suitable assumptions. 
Using variational techniques in the light of Trudinger-Moser inequality, the authors showed the existence of a weak solution to \eqref{r}.\\
We refer to the works \cite{CE, DGN, EG, GR} involving log type singular nonlinearities. Recent works in connection to the problem $(\mc P_\la)$ have investigated problems of the type 
\begin{equation}\label{1.5}
\left\{	-\De u= (g(u) +\la h(x,u))\chi_{\{ u>0\}}~\text{in}~\Om,~u=0~\text{on}~\pa\Om,\right.
\end{equation}
where $\Om$ is a smooth bounded domain in $\mb R^n$, $n \geq 2 $, $g$ is a singular function either of the type $(a)$ $-t^{-\ba}$, $0 <\ba<1$ or of the type $(b)$ $\log t$, $t>0$ and $h$ is a regular function having different growth rates. Equation \eqref{1.5} with $g(t)$ having type $(a)$ singularity and $h(x,t)=t^p$ with $0<p<(n+2)/(n-2)$, $n \geq 3$ was studied in \cite{DM1,DM2,SY}. Moreover, \cite{LM,MQ} studied \eqref{1.5} with $g(t)$ of type $(b)$ and $h$ as above in dimension $n \geq 3$. When $n=2$, problems with nonlinearities of exponential growth have been considered. In this direction, the case of $g$ of type $(b)$ and $h$ to be of exponential growth was studied in \cite{FMS}. In this paper, the authors showed the existence of a weak solution under suitable assumptions on $h$ and for small values of $\la$. Lastly the case of $g$ of type $(a)$ and $h$ to be of exponential growth was handled in \cite{S}. Precisely, the author considered the following problem
\begin{equation*}
\left\{-\De u = -u^{-\ba}\chi_{\{ u>0\}} +\la u^p +\mu f(u)~\text{in}~\Om,~u \not \equiv 0~\text{in}~\Om,~u=0 ~\text{on}~\pa \Om,\right.
\end{equation*}
where $p>0$ and showed that the above problem has a nonnegative solution for $\la \geq 0$ when the parameter $\mu >0$ is small.\\
Motivated by the above discussion, we aim to study the nonlocal problem $(\mc P_\la)$. 
In the present paper, in frame of $(\mc P_\lambda)$ we focus on the interaction between an absorption singular nonlinearity and a nonlocal term of Hartree type. Precisely we considered in the right hand side of the equation the competition between combined singularities of types $(a)$ and $(b)$ with exponential nonlinearity of Choquard type. Up to our knowledge, this interaction has not been investigated in previous contributions and can not be tackled by using monotone methods since maximum principle fails due to the absorption term. Therefore, we first introduce the approximated problem $(\mc P_{\e,\la})$ and prove that it admits a weak solution for suitable values of $\e$ and $\la$ using the Mountain pass theorem. Moreover we prove uniform estimates independently of $\e$. Precisely, these solutions are shown to be uniformly bounded in $L^\infty(\Om)$ and $H_0^{1}(\Om)$ independently of $\e$. To this aim, we use an original argument providing the uniform boundedness of Palais-Smale sequences independently of $\e$. The uniformly bounded sequence of approximated solutions then converges weakly as $\e \ra 0^+$ to a nontrivial function. The crucial final step to prove that the limit function is the required nontrivial weak solution of $(\mc P_\la)$ appeals a delicate uniform gradient estimate of the approximated solutions (see Lemma \ref{l3.1} in Section \ref{GE}) recalling some seminal ideas used in \cite{Ph} for singular heat equations.\\
\textbf{Structure of the paper:} In Section \ref{P} we state and prove some preliminary estimates. In Section \ref{suap} we obtain weak solutions $u_\e$ of the approximated problem $(\mc P_{\e,\la})$ using the Mountain Pass theorem. These solutions are shown to be bounded in $H_0^{1}(\Om)$ by a constant $M$ that does not depend upon $\e$. Using this, we also obtain a priori pointwise estimates for solutions $u_\e$. We complete this section by showing that solutions $u_\e$ converge weakly in $H_0^{1}(\Om)$ to a nontrivial function $u$. In Section \ref{GE} we establish the gradient estimate for solutions $u_\e$. In Section \ref{las} we prove Theorem \ref{tmr}.\\
\textbf{Notations:} The following notations will be used throughout the paper.
\begin{itemize}
\item The energy norm of the space $H_0^{1}(\Om)$ is denoted by $\| \cdot \|$.
\item $| \cdot |_p$ will denote the usual norm of $L^p(\Om)$ space.
\item $C$ will denote generic constant that may vary from line to line.
\end{itemize}

		\section{Preliminaries}\label{P}
		\setcounter{equation}{0}
		\noi The study of elliptic equations with exponential growth nonlinearities
		are motivated by the following Trudinger-Moser inequality, namely
		
		\begin{Theorem}\label{t2.1}
			We have, 
			\begin{equation*}
				\exp\left(\al w^{2}\right)\in L^1(\Om)~\text{ for every}~ w \in H^{1}_0(\Om)~ \text{and}~ \al > 0,
			\end{equation*}
			
			and there is a constant $k_1 > 0$ such that
			\begin{equation}\label{eq2.3}
				\sup\limits_{\|w\|\leq 1}\int\limits_{\Om}\exp\left(\al w^{2}\right)\leq k_1,~\text {for every} ~\al \leq 4\pi~\text{and}~ w\in H^{1}_0(\Om).		
			\end{equation}
		\end{Theorem}	
		The embedding $H_0^{1}(\Om)\ni u \mapsto \exp\left(|u|^\beta\right)\in L^1(\Om)$ is compact for all $\beta	\in [1,2)$ and is continuous for $\beta =2$. Consequently the map $T : H_0^{1}(\Om) \ra L^q(\Om)$, for $q \in [1, \infty),$ defined by 
		$	T(u) := \exp \left(|u|^{\beta}\right)$, is continuous with respect to the norm topology, for any $\beta   \leq 2$.\\
		The following result due to Lions \cite[Section 1.7, Remark I.18]{L} will play an important role.
		\begin{Lemma}\label{lL}
			Let $\{u_k\}$ be a sequence of functions in $H_0^{1}(\Om)$ with $\|u_k \|=1$ such that $u_k \rightharpoonup u \ne 0$ weakly in $H_0^{1}(\Om)$. Then for every 
			\begin{equation*}
				0<p <\frac{4\pi}{(1-\|u\|^2)}
			\end{equation*}
			we have
			\begin{equation*}
				\sup_k\I{\Om}\exp\left(pu_k^{2}\right)<a_2,~\text{for some constant}~a_2~\text{independent of}~k.
			\end{equation*}
		\end{Lemma}
		Now note that using hypothesis $(f_2)$, we can conclude that for any $\alpha>0$ there exists a constant $C>0$ depending on $\al$ such that
		\begin{equation}\label{e2.3}
			\max \{|f(t),|F(t)|\} \leq C \exp\left(\al t^{2}\right),~\text{for}~t \in \mb{R}.
		\end{equation}
		For any $u \in H_0^{1}(\Om)$, by virtue of Sobolev embedding theorem, we get that $u \in L^q(\Om)$ for all $q \in [1,\infty)$. Moreover using Trudinger-Moser inequality and \eqref{e2.3}, we get 
		\begin{equation}\label{e2.4}
		f(u), F(u)\in L^q(\Om)~\text{for any} ~q\geq 1.
		\end{equation}
		Now we recall the well known Hardy-Littlewood-Sobolev inequality
		\begin{Proposition}\label{hls}
			\cite{hls}\textbf{Hardy-Littlewood-Sobolev inequality} Let $r, s>1$ and $0<\mu<2$ with $1/r+1/s+\mu/2=2$, $f\in L^{r}(\mb R^2), g \in L^s(\mb R^2)$. Then, there exist a sharp constant $C(r,s,\mu)$ independent of $f$ and $g$ such  that 
			\begin{equation}\label{hlse}
				\int\limits_{\mb R^2}\int\limits_{\mb R^2}\frac{f(x)g(y)}{|x-y|^{\mu}}dxdy \leq C(r,s,\mu) |f|_r|g|_s.
			\end{equation}\QED
		\end{Proposition}
		By taking $r=s=4/(4-\mu)$ in Proposition \ref{hls} and using \eqref{e2.4}, we get that $J_{\e,\la}$ as defined in \eqref{eq2.1} is well defined in $H_0^{1}(\Om)$. Also $J_{\e,\la} \in C^1(H_0^{1}(\Om),\mb R)$. Naturally the critical points of $J_{\e,\la}$ corresponds to weak solutions of $\mc (P_{\e,\la})$ and for any $u \in H_0^{1}(\Om)$
		\begin{equation*}
			\ld J_{\e,\la}^{\prime}(u),\varphi\rd = \I{\Om}\na u \na \varphi dx + \I{\Om}l_\e(u)\varphi dx -\I{\Om}\left( \I{\Om}\frac{F(u(y))}{|x-y|^\mu}dy\right)f(u(x))\varphi(x)dx,
		\end{equation*}
		for every $\varphi \in H_0^{1}(\Om)$. We will need the following estimates of $l_\e$.
\begin{Proposition} The following assertions hold.
	\begin{enumerate}
		\item There exists $\tilde m>0$ such that
		\begin{equation}\label{2.6}
			0\leq- l_\e(t)\leq t~ \text{for all} ~t\geq 1-\e~\text{and}~0<\e <1/2,
		\end{equation}
		\begin{equation}\label{2.7}
		0<	L_\e(t)< \widetilde{m}~\text{for any}~0\leq t\leq 1-\e~\text{and}~0<\e<1,
		\end{equation}
	and
	\begin{equation}\label{2.8}
		|L_\e(t)|\leq \frac{t^{2-\ba}}{2-\ba}+\frac{t^{1-\ba}}{1-\ba}+\widetilde{m}.
	\end{equation}
\item  For each $p_0 >2$ there exists a constant $ k_0$ such that
\begin{equation}\label{2.9}
	L_\e (t) \geq -k_0t^{p_0}~\text{for all}~t \geq 1-\e,~0\leq \e \leq 1/2.
\end{equation}
\item There exists a constant $C>0$ that does not depend upon $0<\e<1$ such that
\begin{equation}\label{2.10}
	|t l_\e(t)| \leq C(1+t^{2-\ba})~\text{for all}~ t>0.
\end{equation}
\item There exist constants $\e_0 >0$ and $\de_0>0 $ (that is independent of $\e>0$) such that
\begin{equation}\label{2.11}
	-\log\left(t+\frac{\e}{t+\e}\right) \geq t~\text{for all}~0\leq t<{\de_0}~\text{and}~0<\e<\e_0.
\end{equation}
	\end{enumerate}
\end{Proposition}
\begin{proof} By definition of $l_\e$, we have $l_\e(t)\geq 0$ for $0 \leq t \leq 1-\e$ and $l_\e(t)\leq 0$ for $t \geq 1-\e$. Note that for $0<\e<1/2$ and $t \geq 1-\e$, we have
	\begin{align*}
		-l_\e(t)=\frac{t^q}{(t+\e)^{q+\beta}}\log\left(t+\frac{\e}{t+\e}\right) \leq t(t+\e)^{-\ba-1}\log(t+\e)\leq t.
	\end{align*}
This proves assertion \eqref{2.6}. \\
 Now since $0<\ba<1$, we choose $ \de>0$ such that $0<\ba<\de<1$. Since $\lim_{t \ra 0} -t^{\de -\ba}\log(t)=0$, we can choose $m>0$ such that $-t^{\de-\ba}\log(t)<m$ for $0\leq t\leq 1$. Hence for $0\leq t \leq 1-\e$, we have
	\begin{align*}
		0\leq l_\e(t)\leq -(t+\e)^{-\ba}\log(t+\e)\leq m t^{-\de},
	\end{align*}
and so 
\begin{align*}
	0<L_\e(t)=\I{0}^{t}l_\e(s)ds\leq \frac{m}{1-\de}t^{1-\de}\leq \widetilde{m}~\text{for all}~0\leq t\leq 1-\e.
\end{align*}
This concludes assertion \eqref{2.7}. Inequality \eqref{2.8} also holds. Indeed using \eqref{2.7} and the fact that $\log t \leq t$ for all $t >0$, we have
\begin{align*}
	|L_\e(t)| \leq& \widetilde{m} +\I{1-\e}^{t}|l_\e(s)|ds \leq \widetilde{m} +\I{1-\e}^{t}s^{-\ba} \log\left(s+\frac{s}{s+\e}\right)\\
	\leq & \widetilde{m} +\I{1-\e}^{t}s^{-\ba}\left(s+\frac{\e}{s+\e}\right)ds \leq \widetilde{m} + \I{0}^{t}(s^{1-\ba}+ s^\ba)ds =\widetilde{m} +\frac{t^{2-\ba}}{2-\ba}+\frac{t^{1-\ba}}{1-\ba}.
\end{align*}
Note that for each $p_0>2$ there exists $k_0 >0$ such that 
\begin{equation*}
	\widetilde{m} +\frac{t^{2-\ba}}{2-\ba}+\frac{t^{1-\ba}}{1-\ba} \leq k_0 t^{p_0}~\text{for all} ~t \geq 1/2.
\end{equation*}
Thus, from \eqref{2.8} we obtain
\begin{align*}
L_\e(s)\geq - \left(\widetilde{m} +\frac{t^{2-\ba}}{2-\ba}+\frac{t^{1-\ba}}{1-\ba}\right) \geq -k_0t^{p_0}~\text{for all}~t\geq 1-\e~\text{and for every}~0<\e<1/2,
\end{align*}
proving \eqref{2.9}. Now we prove \eqref{2.10}. For each $ 0<\e<1$ and $0<t<1-\e$ there exists a constant $C>0$ independent of $\e$ such that
\begin{align*}
	|l_\e(t)t| \leq -t^{1-\ba}\log\left(t +\frac{\e}{t+\e}\right) \leq (-\log t)t^{1-\ba}\leq C.
\end{align*}
On the other hand for $t >1-\e$ we have for some $C>0$
\begin{align*}
	|l_\e(t)t| \leq t^{1-\ba}\log\left(t +\frac{\e}{t+\e}\right) \leq t^{1-\beta}\log(t+\e)\leq Ct^{2-\beta}.
\end{align*}
We conclude that there exists a constant $C>0$ such that 
\begin{equation*}
	|l_\e(s)s| \leq C(1+t^{2-\ba})~\text{for all}~s \geq 0~\text{and}~0<\e<1.
\end{equation*}
Lastly, we prove assertion \eqref{2.11}. A simple application of L'Hospital's rule shows that
\begin{equation*}
	\lim\limits_{t \ra 0} \frac{t}{-\log\left(t+\frac{\e}{t+\e}\right)} =\frac{\e}{1-\e} \leq \frac{1}{2}~\text{if}~0<\e<\e_0 =\frac{1}{3}.
\end{equation*}
Thus we can choose $\de_0$ independent of $0<\e<\e_0$  such that
\begin{equation*}
	 t \leq -\log\left(t+\frac{\e}{t+\e}\right) ~\text{if}~0\leq t <\de_0~\text{and}~0<\e<\e_0.
\end{equation*}
This completes the proof. \QED
	\end{proof}
	We shall require the following proposition many times in the course of this article.
\begin{Proposition}\label{p2.6}
	Assume that $f$ satisfies \eqref{e2.3} for some $C,\al>0$. If $\{u_k\}\subset H_0^{1}(\Om)$ is a sequence such that $\|u_k\|^{2}\leq\frac{2\pi}{\al}$ and $u_k \rightharpoonup u$ weakly in $H_0^{1}(\Om)$, then
	\begin{equation}\label{e2.15}
		\I{\Om}f(u_k)dx \ra \I{\Om}f(u)dx
	\end{equation}	
	and \begin{equation}\label{e2.16}
		\I{\Om}F(u_k)dx\ra \I{\Om}F(u)dx.
	\end{equation}
	In particular, if $f$ satisfies $(f_2)$, then the assertions \eqref{e2.15} and \eqref{e2.16} hold provided that there is a constant $M>0$ that does not depend upon $k$ such that $\|u_k\|<M$.
\end{Proposition}
\begin{proof}
	Proof follows similarly as the proof of Proposition $2$ in \cite{FMS}. \QED
\end{proof}
\section{Solution and uniform bounds of approximated problem} \label{suap}
		The following theorem is the main result concerning $ (\mc  P_{\e,\la})$ in this section:
		\begin{Theorem}\label{t2.2}
			Suppose $f$ satisfies $(f_2)-(f_5)$ and let $\e^\ast =\min\{\e_0,1/2\}$, where $\e_0$ is defined as in \eqref{2.11}. Then there exists a $\bar{\la} >0$ such that for each $0 <\la <\bar{\la}$ we have a nonnegative nontrivial solution $u_\e \in H_0^1(\Om)$ to ${(\mc P_{\epsilon,\la})}$ for every $0<\epsilon<\e^\ast$. Moreover there exists a constant $M>0$, independent of $\e\in(0,\epsilon^*)$ and $\lambda\in(0,\bar{\lambda})$ such that 
			\begin{equation}
				\|u_\e\| \leq M ~\text{and}~|u_\e|_{\infty}<M.
			\end{equation}
		\end{Theorem}	
		
		\noi We prove Theorem \ref{t2.2} with the help of mountain pass lemma.\\
		First we show that $J_{\e,\la}$ satisfies the Palais-Smale condition at every nonzero level $c$.
		\begin{Lemma}\label{lb}
			Fix $0<\e<1$ and suppose $f$ satisfies $(f_2)$, $(f_4)$ and $(f_5)$. Then the functional $J_{\e,\la}$ satisfies the Palais-Smale condition at every level $c \ne 0$.
		\end{Lemma}
		\begin{proof}
			Let $\{u_k^{\e}\}_{k\in\mathbb{N}}$ be a Palais-Smale sequence for $J_{\e,\la}$ in $H_0^{1}(\Om)$ at level $c$. Throughout this proof we shall denote $u_k^\e$ by $u_k$. Then $\{u_k\}_{k\in\mathbb{N}}$ satisfies 
			\begin{equation}\label{eq2.14}
				\frac{\|u_k\|^2}{2}+\int\limits_{\Om}L_\e(u_k)dx-\frac{\la}{2} \I{\Om}\left( \I{\Om}\frac{F(u_k(y))}{|x-y|^\mu}dy\right)F(u_k(x))dx \ra c~\text{as}~k\ra \infty,
			\end{equation}	
			and there is a sequence $\rho_k \ra 0$ such that
			\begin{equation}\label{eq2.15}
				\left| \int\limits_{\Om}\na u_k \na w +\int\limits_{\Om}l_\e(u_k)w-\la \I{\Om}\left( \I{\Om}\frac{F(u_k(y))}{|x-y|^\mu}dy\right)f(u_k(x))w(x)dx\right|\leq \rho_k \|w\|,
			\end{equation}	
			for all $w \in H_0^{1}(\Om)$.	
			Taking $w=u_k$ in \eqref{eq2.15}, we have
			\begin{equation}\label{e2.19}
				\left| \|u_k\|^2 +\int\limits_{\Om}l_\e(u_k)u_k-\lambda\I{\Om}\left( \I{\Om}\frac{F(u_k(y))}{|x-y|^\mu}dy\right)f(u_k(x))u_k(x)dx\right| \leq \rho_k \|u_k\|.
			\end{equation}	
			From the assumption $(f_4)$, there exists $\zeta >1$ such that $\zeta F(t) \leq tf(t)$ for any $t>0$, which yields
			\begin{equation*}
				\zeta \I{\Om}\left( \I{\Om}\frac{F(u(y))}{|x-y|^\mu}dy\right)F(u(x))dx \leq \I{\Om}\left( \I{\Om}\frac{F(u(y))}{|x-y|^\mu}dy\right)f(u(x))u(x)dx.
			\end{equation*}
			Then 
			\begin{align}\nonumber \label{e2.20}
				J_{\e,\la}(u_k)-&\frac{1}{2\zeta}\ld J_{\e,\la}^{\prime}(u_k),u_k\rd\\\nonumber
				&=\left(\frac{1}{2}-\frac{1}{2\zeta}\right)\|u_k\|^2 + \I{\Om}L_\e(u_k)-\frac{1}{2\zeta}\I{\Om}l_\e(u_k)u_k\\ \nonumber
				-&\frac{\la}{2}\left(\I{\Om}\left( \I{\Om}\frac{F(u_k(y))}{|x-y|^\mu}dy\right)F(u_k(x))dx-\frac{1}{\zeta}\I{\Om}\left( \I{\Om}\frac{F(u_k(y))}{|x-y|^\mu}dy\right)f(u_k(x))u_k(x)dx\right)\\
				&\geq \left(\frac{1}{2}-\frac{1}{2\zeta}\right)\|u_k\|^2 + \I{\Om}L_\e(u_k)-\frac{1}{2\zeta}\I{\Om}l_\e(u_k)u_k.
			\end{align}
			Now we estimate the second and the third terms in \eqref{e2.20} independent of $\e$. Using Sobolev embedding theorem, we have for sufficiently large $K$ and $\beta \in (0,1)$
			\begin{align}\nonumber\label{e2.21}
				\I{\Om}L_\e(u_k) =& \I{\Om}\I{0}^{u_k} -\frac{t^q}{(t+\e)^{q+\ba}}\log\left(t+\frac{\e}{t+\e}\right)dtdx
				\geq \I{\Om}\I{1-\e}^{u_k} -t^{-\ba}\log\left(t+\frac{\e}{t+\e}\right)dtdx\\
				\geq & -\I{\Om}\I{0}^{u_k} t^{-\ba}\log\left(t+K\right)dtdx
				\geq -\I{\Om}u_kdx-C \geq -C\|u_k\|-C.
			\end{align}
			Again using Sobolev embedding and H\"{o}lder's inequality, we have for $\gamma \in (0,1-\ba)$
			\begin{align}\nonumber\label{e2.22}
				-\I{\Om}l_\e(u_k)u_k=&\I{\Om}\frac{u_k^q}{(u_k+\e)^{q+\ba}}\log\left(u_k+\frac{u_k}{u_k+\e}\right)u_k
				\geq\I{\{0<u_k<1-\e\}}u_k^{1-\ba}\log\left(u_k+\frac{u_k}{u_k+\e}\right)\\
				\geq& \I{\{0<u_k<1-\e\}}\log(u_k)u_k^{1-\ba}
				\geq  -\I{\Om}u_k^{1-\beta-\gamma}dx\geq -C\|u_k\|-C.
			\end{align}
			Plugging \eqref{e2.21} and \eqref{e2.22} in \eqref{e2.20}, we have
			
			\begin{equation}\label{e2.23}
				J_{\e,\la}(u_k)-\frac{1}{2\zeta}\ld J_{\e,\la}^{\prime}(u_k),u_k\rd\geq \left(\frac{1}{2}-\frac{1}{2\zeta}\right)\|u_k\|^2-C\|u_k\|-C.
			\end{equation}
			Also from \eqref{eq2.14} and \eqref{e2.19}, we have
			\begin{equation}\label{e2.24}
				J_{\e,\la}(u_k)-\frac{1}{2\zeta}\ld J_{\e,\la}^{\prime}(u_k),u_k\rd \leq C\left(1 +\rho_k \frac{\|u_k\|}{2 \zeta}\right),
			\end{equation}
			for some constant $C>0$. Therefore from \eqref{e2.23} and \eqref{e2.24}
			\begin{equation*}
				\left(\frac{1}{2}-\frac{1}{2\zeta}\right)\|u_k\|^2-C\|u_k\|-C \leq  C\left(1 +\rho_k \frac{\|u_k\|}{2 \zeta}\right).
			\end{equation*}
			This implies that $\{u_k\}$ must be bounded in $H_0^{1}(\Om)$. Also from the calculations it is clear that we can find $M>0$ independent of $\e$ and $\lambda$ such that 
			\begin{equation}\label{eM}
				\|u_k\| \leq M.
			\end{equation} Consequently, there exists $u \in H_0^{1}(\Om)$ such that up to subsequences, we have
			\begin{align}\label{ewc}
				\begin{cases}
					u_k \rightharpoonup u~\text{weakly in}~H_0^{1}(\Om),\\
					u_k \ra u~\text{in}~L^p(\Om)~\text{for every}~p\geq 1,\\
					u_k(x)\ra u(x)~\text{a.e. in }~\Om,\\
					|u_k|\leq z_p~\text{a.e. in }~\Om~\text{for some}~z_p \in L^p(\Om)\,\mbox{with any } p\in [1,\infty).
				\end{cases}	
			\end{align}
			Now using Proposition \ref{p2.6}, we have $l_\e(u_k) \ra l_\e(u)$ and $L_\e(u_k) \ra L_\e(u)$ in $L^1(\Om)$. 
			Next we show that 
			\begin{equation}\label{e2.27}
				\left(\I{\Om}\frac{F(u_k(y))}{|x-y|^\mu}dy\right)f(u_k(x))dx \ra \left(\I{\Om}\frac{F(u(y))}{|x-y|^\mu}dy\right)f(u(x)) ~\text{in } L^1(\Om).
			\end{equation}
			We have
			\begin{align*}
				&\I{\Om}\left|\left(\I{\Om}\frac{F(u_k(y))}{|x-y|^\mu}dy\right)f(u_k(x))-\left(\I{\Om}\frac{F(u(y))}{|x-y|^\mu}dy\right)f(u(x))\right|\\
				\leq&\I{\Om}\left(\I{\Om}\frac{F(u(y))}{|x-y|^\mu}dy\right)|f(u_k(x))-f(u(x))|dx\\
				&+\I{\Om}\left(\I{\Om}\frac{|F(u_k(y))-F(u(y))|dy}{|x-y|^\mu}\right)f(u_k(x))dx\\
				 :=& I_1 +I_2 ~\text{(say)}.
			\end{align*}
			From \eqref{e2.4}, we know that $F(u) \in L^r(\Om)$ for any $r \in [1,\infty)$. Since $\mu \in (0,1)$, $y \mapsto |x-y|^{-\mu} \in L^{s}(\Om)$ for all $s \in \left(1,\frac{1}{\mu}\right)$ uniformly in $x \in \Om$ (since $\Om$ is bounded). So using H\"{o}lder's inequality we get that 
			\begin{equation}\label{e2.26}
				\I{\Om}\frac{F(u(y))}{|x-y|^\mu}dy \in L^{\infty}(\Om).
			\end{equation}
			Now by  using \eqref{e2.26} and Proposition \ref{p2.6}, we get that $I_1 \ra 0$ as $k \ra \infty$. Now using the semigroup property of Riesz potential, we have 
			\begin{align}\nonumber
			&	\I{\Om}\left(\I{\Om}\frac{|F(u_k(y)-F(u(y)))|}{|x-y|^\mu}dy\right)f(u_k(x))dx\\ \nonumber
				\leq& \left( \I{\Om}\left(\I{\Om}\frac{|F(u_k(y))-F(u(y))|}{|x-y|^\mu}dy\right)|F(u_k(x))-F(u(x))|dx\right)^\frac{1}{2}\\
				&  \left( \I{\Om}\left(\I{\Om}\frac{f(u_k(y))}{|x-y|^\mu}dy\right)f(u_k(x))dx\right)^\frac{1}{2}.
			\end{align}
			Again by Hardy-Littlewood-Sobolev inequality and \eqref{e2.3}, we have
			\begin{align*}
				\left( \I{\Om}\left(\I{\Om}\frac{f(u_k(y))}{|x-y|^\mu}dy\right)(f(u_k(x)))dx\right)^\frac{1}{2} \leq& \left( \I{\Om}|f(u_k)|^{\frac{4}{4-\mu}}\right)^\frac{4-\mu}{4}\\
				\leq & C\left( \I{\Om}\exp\left(\frac{4}{4-\mu}\al u_k^{2}\right)\right)^\frac{4-\mu}{4}.
			\end{align*}
			Choosing $\al >0$ such that $\frac{4}{4-\mu}\al M^{2} \leq 4\pi$, where $M$ is given by \eqref{eM}, we obtain using \eqref{eq2.3} that
			\begin{equation*}
				\left( \I{\Om}\left(\I{\Om}\frac{f(u_k(y))}{|x-y|^\mu}dy\right)(f(u_k(x)))dx\right)^\frac{1}{2} \leq \widetilde{C},
			\end{equation*}
			where $\widetilde{C}$ is a positive constant. Next 
			\begin{equation*}
				\left( \I{\Om}\left(\I{\Om}\frac{F(u_k(y))-F(u(y))}{|x-y|^\mu}dy\right)(F(u_k(x)))-F(u(x))dx\right) \ra 0
			\end{equation*}
			follows similarly as in the proof of Lemma $3.4$ of \cite{AGMS}. Finally, the convergence
			\begin{equation}\label{e2.30}
				\left(\I{\Om}\frac{F(u_k(u))}{|x-y|^\mu}dy\right) F(u_k(x))\ra 	\left(\I{\Om}\frac{F(u(u))}{|x-y|^\mu}dy\right) F(u(x))~\text{in}~L^1(\Om)
			\end{equation}
			follows from $(f_5)$ and Vitali's convergence Theorem  (see for instance the proof of Lemma $2.4$ in \cite{ACT}). From now onwards in this proof we use the following notations
			\begin{equation*}
				P(x,u)=\frac{\la}{2}\left(\I{\Om}\frac{F(u(y))}{|x-y|^\mu}dy\right) F(u(x)) -L_{\e}(u)~\text{and}~ R(x,u)=\la \left(\I{\Om}\frac{F(u(y))}{|x-y|^\mu}dy\right)f(u)-l_\e(u).
			\end{equation*}
			From \eqref{eq2.14}, \eqref{e2.19}, \eqref{e2.27} and \eqref{e2.30}, we have
			\begin{equation}\label{e2.31}
				\lim\limits_{k \ra \infty} \|u_k\|^2 = 2\left(c + \I{\Om}P(x,u)dx \right)
			\end{equation}
			and 
			\begin{equation}\label{e2.32}
				\lim\limits_{k \ra \infty} \I{\Om}R(x,u_k)u_k =2\left(c + \I{\Om}P(x,u)dx \right).
			\end{equation}
			Now, we consider two cases:\\
			\textbf{Case 1:} Suppose that $c \ne 0$ and $u \ne 0$. One has $J_{\e,\la}(u)\leq c$. We claim that $J_{\e,\la}(u)=c$. Assume by contradiction that $J_{\e,\la}(u) <c.$
			Then
			\begin{equation}\label{e2.33}
				\|u\|^2 < 2\left(c +\I{\Om}P(x,u)dx\right).
			\end{equation}
			Let $v_k =\ds \frac{u_k}{\|u_k\|}$ and $\ds v = \frac{u}{\left( 2\left(c + \I{\Om}P(x,u)dx \right)\right)^{1/2}}$. It follows that $v_k \rightharpoonup v$ weakly in $H_0^{1}(\Om)$, $\|v_k\|=1$ and $\|v\|<1$. Hence by Lemma \ref{lL}, we have
			\begin{equation}\label{e2.34}
				\sup\limits_{k} \I{\Om}\exp\left(4\pi pv_k^{2}\right)\leq k_1=k_1(p)<\infty\text{ for any } 1<p<(1-\|v\|^2)^{-1}.
			\end{equation}
			Now for $q>1$ small enough, we have
		\begin{align}\nonumber\label{eC}
				\I{\Om}|R(x,u_k)|^q \leq& C \I{\Om}\left[\left(\I{\Om}\frac{|F(u_k(y))f(u_k(x))dy}{|x-y|^\mu}\right)^q+|l_\e(u_k(x))|^q\right]dx\\ \nonumber
				\leq& C\left(\I{\Om}\I{\Om}\frac{\exp\left(q\al u_k^{2}(y)\right)\exp \left(q\al u_k^{2}(x)\right)}{|x-y|^{q\mu}}dxdy +1\right)\\
				\leq& C \left(\I{\Om}\exp\left(\frac{4}{4-q\mu}q\al u_k^{2}\right)\right)^\frac{4-q\mu}{2}+C.
			\end{align}
			Now we choose $\al$ and $p$ such that 
			\begin{equation*}
				\frac{4}{4-q\mu}q\al\|u_k\|^{2} <p <\frac{4\pi}{(1-\|v\|^2)}.
			\end{equation*}
			Hence fix $p>0$ such that $\ds p <\frac{4\pi}{(1-\|v\|^2)}$ and take $\al >0$ small such that $$\ds \al <\frac{4-q\mu}{4}\frac{p}{qM^{2}},$$ where $M$ is given by \eqref{eM}. Then by \eqref{e2.34}, we have 
			\begin{equation}\label{e2.35}
				\I{\Om}|R(x,u_k)|^q \leq C, ~\text{for some}~C>0.
			\end{equation}
	Using \eqref{ewc}, \eqref{e2.32}, \eqref{e2.35}, and H\"{o}lder's inequality we have
			\begin{align}\label{e2.36}
				2\left(c+\I{\Om}P(x,u)dx\right)=\lim\limits_{k \ra \infty}\I{\Om}R(x,u_k)u_kdx =\lim\limits_{k \ra \infty} \I{\Om}R(x,u_k)udx.
			\end{align}
			On the other hand from \eqref{e2.15} 
			\begin{align*}
				\left| \I{\Om}\na u_k\na u -\I{\Om}R(u_k)u\right|\leq \rho_k\|u\| ~\text{for each}~k \in \mb N.
			\end{align*}
			Hence,
			\begin{align}\label{e2.38}
				-\rho_k\|u\| + \I{\Om}R(x,u_k)u\leq \I{\Om}\na u_k \na u\leq \rho_k\|u\|  + \I{\Om}R(x,u_k)u.
			\end{align}
			Taking $k \ra \infty$ in \eqref{e2.38} and using \eqref{e2.36}, we get 
			\begin{equation*}
				\|u\|^2= 2 \left(c+\I{\Om}P(x,u)dx\right),
			\end{equation*}
			which is a contradiction to \eqref{e2.33}. 
			Therefore, we must have $J_{\e,\la}(u)=c$. As a consequence, using \eqref{e2.31}, we obtain
			\begin{equation*}
				\lim\limits_{k \ra \infty} \|u_k\|^2 = 2\left(c +\I{\Om}P(x,u)dx \right) =\|u\|^2.
			\end{equation*}
			Thus, it follows that $u_k \ra u$ strongly in $H_0^{1}(\Om)$.\\
			\textbf{Case 2:} Assume that $c \ne 0$ and $u =0$. We will prove that this cannot happen. Fix a constant $0<a<1$. From $ u =0$, we conclude from \eqref{e2.31} that  for large $k$
			\begin{equation}
				\|u_k\|^2 \leq 2c+a.
			\end{equation}
			According to \eqref{eC}, take $q>1$ small enough such that $ \ds\al < \frac{4-\mu}{4}\frac{4\pi}{q(2c+a)} $. Then,
			\begin{equation*}
				\I{\Om}|R(x,u_k)|^q <C,
			\end{equation*}
			where $C$ is some positive constant. Now using H\"{o}lder's inequality and the fact that $u_k \ra 0$ strongly in $L^{\frac{q}{q-1}}(\Om)$, we get
			\begin{align}
				\left|	\I{\Om}R(x,u_k)u_k\right| \leq \left( \I{\Om}|R(x,u_k)|^q\right)^\frac{1}{q}\left( \I{\Om}|u_k|^{\frac{q}{q-1}}\right)^\frac{q-1}{q}\ra 0 ~\text{as}~k \ra \infty.
			\end{align}
			On the other hand from \eqref{e2.19}, we have
			\begin{equation*}
				\left| \I{\Om}|\na u_k|^2 - \I{\Om}R(x,u_k)u_k\right| \leq \rho_k\|u_k\|.
			\end{equation*}
			This means that 
			\begin{align*}
				-\rho_k M +\I{\Om}R(x,u_k)u_k \leq \I{\Om}|\na u_k|^2dx \leq \rho_kM +\I{\Om}R(x,u_k)u_k.
			\end{align*}
			Taking limit as $ k\ra \infty$ above, we get $\|u_k\| \ra 0$. This contradicts the fact that $\|u_k\| \ra 2c \ne 0$. The proof is complete. \QED
		\end{proof}
		Let $ \varphi \in H_0^{1}(\Om)\cap L^{\infty}(\Om)$ be such that $\varphi >0$ in $\Om$ and $\|\varphi\| =1$.

		\begin{Lemma}\label{lmp}
			Suppose $f$ satisfies \eqref{1.3} and take $\la >0$. Then there exists a constant $K =K(\la)$ independent of $\epsilon$ such that 
			\begin{equation*}
				J_{\e,\la}(K \varphi) <0.
			\end{equation*}
			Also there exists a constant $m_2 >0$, which is independent of $\e$, such that
			\begin{equation*}
				\sup\limits_{t \in [0,1]} J_{\e,\la}(tK \varphi)<m_2.
			\end{equation*}
		\end{Lemma}
		\begin{proof} Since $l_{\e}(t)\leq 0$ for all $t\geq 1-\e$, we conclude using \eqref{2.7} that 
			\begin{equation}\label{eqf}
				L_{\e}(t)\leq \widetilde{m}~\text{ for all} ~t \geq 0~\text{and}~\e>0.
			\end{equation}
			Choose a compact set $\widetilde{\Om} \subset \Om$ such that $\varphi > a_0 >0$ on $\widetilde{\Om}$ for some $a_0>0$. Then using \eqref{1.3} and \eqref{eqf}, for $t$ sufficiently large, we have
			\begin{align*}
				J_{\e,\la}(t \varphi) \leq& \frac{t^2}{2}+\widetilde{m}|\Om|- \la \I{\widetilde{\Om}}\I{\widetilde{\Om}}\frac{F(t\varphi(y))F(t\varphi(x))}{|x-y|^\mu}dxdy\\
				\leq & \frac{t^2}{2} +\widetilde{m}|\Om| -\la \I{\widetilde{\Om}}\I{\widetilde{\Om}}\frac{A(t \varphi)^\gamma(y)A(t \varphi)^\gamma(x)}{|x-y|^\mu}dxdy\\
				=& \frac{t^2}{2} +\widetilde{m}|\Om| -A^2\la t^{2\gamma} \I{\widetilde{\Om}}\I{\widetilde{\Om}}\frac{ \varphi^\gamma(y) \varphi^\gamma(x)}{|x-y|^\mu}dxdy\\
				=& \frac{t^2}{2} +\widetilde{m}|\Om| -C(\la) t^{2\gamma} \ra -\infty ~\text{as}~ t \ra \infty.
			\end{align*}			
			Hence, there exists $K >0$ such that $J_{\e,\la}(K\varphi)<0$.\\
			Now we see that 
			\begin{equation*}
				J_{\e,\la}(tK \varphi)=\frac{t^2K^2}{2}+\I{\Om}L_{\e}(tK\varphi)-\la \I{\Om}\I{\Om}\frac{F(tK\varphi(y))F(tK\varphi(y))}{|x-y|^\mu}dxdy.
			\end{equation*}
			Since $0 \leq tK\varphi \leq K \sup\limits_{\Om} \varphi$ for all $0 \leq t \leq 1$, there exists $C_1>0$ depending on $\la$ such that 
			\begin{equation*}
				J_{\e,\la}(tK\varphi)\leq \frac{t^2K^2}{2}+C_1~\text{for all}~0\leq t\leq1.
			\end{equation*}
			Hence,
			\begin{equation*}
				\sup\limits_{t \in [0,1]} J_{\e,\la}(tK\varphi)<m_2,
			\end{equation*}
			where
			\begin{equation*}
				m_2=\frac{K^2}{2}+C_1.
			\end{equation*} 
			This proves Lemma \ref{lmp}.\QED 
		\end{proof}
		Next proposition provides the existence of a critical point $u_\e$ for the functional
		$J_{\e,\la}$ and a priori bounds that hold for these solutions.
		\begin{Proposition}\label{p2.9}
			Suppose $f$ satisfies $(f_2)$, $(f_4)$ and $(f_5)$ and let $m_2$ given by Lemma \ref{lmp}. Then there exist constants $\bar{\la}$, $m_1 >0$ such that for each $0<\la <\bar{\la}$ fixed, problem $ (\mc P_{\e,\la})$ has a weak solution $u_\e$ with $0<m_1<J_{\e,\la}(u_\e)<m_2$. Also there is a constant $M>0$ independent of $\e$ such that 
			\begin{equation*}
				\|u_\e\|\leq M.
			\end{equation*}
		\end{Proposition}
		\begin{proof}
			Let $M>0$ and $u\in H^{1}_0(\Omega)$ such that $||u||\leq M$. Using Hardy-Littlewood-Sobolev inequality and \eqref{e2.3}, we have 	
			\begin{align}\nonumber\label {e2.42}
				\I{\Om}\I{\Om}\frac{F(u(y))F(u(x))}{|x-y|^\mu}dydx \leq& C\left(\I{\Om}|F(u)|^{\frac{4}{4-\mu}}dx\right)^\frac{4-\mu}{2}\\
				\leq & C\left(\I{\Om}\exp\left(\frac{4}{4-\mu}\al u^{2}\right)dx\right)^\frac{4-\mu}{2}.
			\end{align}
			Taking $\al \leq \displaystyle\frac{4-\mu}{ M^{2}}\pi$, we get using \eqref{eq2.3} and \eqref{e2.42}, a positive $\widetilde{C}=\widetilde{C}(M)$ such that
			\begin{equation}\label{e2.43}
				\I{\Om}\I{\Om}\frac{F(u(y))F(u(x))}{|x-y|^\mu}dydx \leq \widetilde{C}.
			\end{equation}
			Now by using \eqref{2.9}, \eqref{e2.43} and Sobolev embedding theorem, we have for $p_0 >2$
			\begin{align} \label{e2.44}
				J_{\e,\la}(u)\geq& \frac{\|u\|^2}{2}-k_0\I{\Om}|u|^{p_0}-\la\widetilde{C}
				\geq \frac{\|u\|^2}{2}-k_0C_{p_0}^{p_0}\|u\|^{p_0}-\la\widetilde{C}.
			\end{align}	
			If we assume that $M \leq \theta$, where 
			$\ds	\theta =\left( \frac{1}{4k_0 C_{p_0}^{p_0}}\right)^\frac{1}{p_0-2},$	
			then from \eqref{e2.44}, we have
			\begin{equation*}
				J_{\e,\la}(u)\geq \frac{\|u\|^2}{4}-\la\widetilde{C}.
			\end{equation*}
			Finally, by taking $\ds \bar{\la} =	\frac{\theta^2}{8\widetilde{C}}$, we obtain
			\begin{equation*}
				J_{\e,\la}(u)\geq \frac{1}{4}\left(\|u\|^2 -\frac{\theta^2}{2} \right)~\text{for}~0<\la<\bar{\la}, ~\|u\|<\theta.
			\end{equation*}
			Thus, we obtain for $0<\lambda<\bar{\la}$
			\begin{equation*}
				J_{\e,\la}(0)=0,~J_{\e,\la}(K\varphi)<0~
				\text{and}~
				J_{\e,\la}(u)\geq m_1~\text{for}~\|u\| =\theta,
			\end{equation*}		
			where \begin{equation*}
				m_1 =\frac{\theta^2}{8}.
			\end{equation*}	
			Let \begin{equation*}
				\Sigma =\{\sigma \in C([0,1],H_0^{1}(\Om)): \sigma(0)=0,~ \sigma(1)=K\varphi\}.
			\end{equation*}		
			By the Mountain Pass theorem \cite[page 12]{W}, we conclude that there is a sequence $\{u_k^\e\}$ in $H_0^{1}(\Om)$ and a number
			\begin{equation*}
				c_\e:= \inf\limits_{\sigma \in \Sigma}\sup\limits_{t \in [0,1]} J_{\e,\la}(\sigma(t))
			\end{equation*} 		
			such that
			\begin{equation*}
				\lim\limits_{k \ra \infty} J_{\e,\la}(u_k^\e)=c_\e~\text{and}~\lim\limits_{k \ra \infty} J_{\e,\la}^{\prime}(u_k^\e)=0,
			\end{equation*}	
			i.e.,
			\begin{equation*}
				\frac{\|u_k^\e\|^2}{2}+\I{\Om}L_\e(u_k^\e)-\I{\Om}\I{\Om}\frac{F(u_k^\e)(y)F(u_k^\e)(x)}{|x-y|^\mu}dxdy \ra c_\e ~\text{as}~k \ra \infty
			\end{equation*}
			and there is a sequence $\rho_k$ such that 
			\begin{equation*}
				\left| \int\limits_{\Om}\na u_k^\e\na w +\int\limits_{\Om}l_\e(u_k^\e)w-\la \I{\Om}\left( \I{\Om}\frac{F(u_k^\e(y))}{|x-y|^\mu}dy\right)f(u_k^\e(x))w(x)dx\right|\leq \rho_k \|w\|.
			\end{equation*}
			It is clear that $c_\e >m_1 >0$. Using Lemma \ref{lmp}, we also obtain 
			\begin{equation*}
				c_\e \leq \sup\limits_{t \in [0,1]} J_{\e,\la}(tK\varphi)<m_2.
			\end{equation*}
			Hence for sufficiently large $k$,
			\begin{equation}\label{e2.45}
				0<m_1\leq J_{\e,\la}(u_k^\e)<m_2.
			\end{equation}
			Arguing as in the proof of Lemma \ref{lb}, we may use \eqref{e2.45} to obtain a constant $M>0$ that does not depend upon $\e$ such that 
			\begin{equation*}
				\|u_k^\e \| <M.
			\end{equation*}
			We conclude there is $u_\e \in H_0^{1}(\Om)$ with $\|u_\e\| <M$ such that 
			\begin{equation*}
				u_k^\e \rightharpoonup u_\e ~\text{weakly in}~H_0^{1}(\Om).
			\end{equation*}
			We know that $\{u_k^\e\}$ is a Palais-Smale sequence at a positive level. It follows from \eqref{e2.45} and Lemma \ref{lb}, that upto a subsequence $u_k^\e \ra u_\e$ strongly in $H_0^{1}(\Om)$. Hence $u_\e$ is weak solution of $ (\mc P_{\e,\la})$. This proves the proposition. \QED	
		\end{proof}
		Next we aim to obtain the uniform $L^{\infty}$ estimates of the weak solutions $u_\e$ independently of $\e$. Precisely, we have the following result:
		\begin{Lemma}\label{l2.10}
	Suppose $f$ satisfies $(f_2)$-$(f_5)$, $0<\la <\bar{\la}$ be fixed and  $\e_0$ and $\de_0$ be given by \eqref{2.11}. Then there exists a constant $M_1 >0$ that does not depend upon $\e$ such that 
			\begin{equation}\label{e2.46}
				|u_\e|_{\infty}<M_1 ~\text{for}~0<\e<\e_0,
			\end{equation}
			where $u_\e$ is the nonnegative solution of $ (\mc P_{\e,\la})$ as obtained in Proposition \ref{p2.9}.
		\end{Lemma}
		\begin{proof}
			Choose $r_0 \in \left(1,\frac{1}{\mu}\right)$ and let $s_0$ be the conjugate exponent of $r_0$. From \eqref{e2.4}, we know that 
			\begin{equation*}
				F(u_\e) \in  L^{s_0}(\Om)~\text{for every} ~0<\e<\e_0.
			\end{equation*}	
			We claim that there exists $ M^{\prime}>0$, which is independent of $\e$, such that 
			\begin{equation}\label{e2.47b}
				|F(u_\e)|_{s_0} <M^{\prime}.
			\end{equation} This is a consequence of $(f_2)$ , Theorem \ref{t2.1} and the fact that 
			\begin{equation}\label{e2.47}
				\|u_\e\| <M ~\text{independent of }~\e.
			\end{equation}	
			Indeed, using \eqref{e2.3}, we have	
			\begin{align*}
				\I{\Om}|F(u_\e)|^{s_0} \leq C\I{\Om}\exp(\al s_0 u_\e^2).
			\end{align*}	
			Using \eqref{e2.47}, we can find $\al >0$, independent of $\e$, such that  
			\begin{equation*}
				\al s_0 \|u_\e\|^2<4\pi.
			\end{equation*}	
			Using this fact and \eqref{eq2.3}, we conclude that \eqref{e2.47b} holds. Also since $\Om$ is bounded, $ y \mapsto |x-y|^{-\mu} \in L^{r_0}(\Om)$, uniformly in $x \in \Om$. Hence using H\"{o}lder's inequality, we have
			\begin{equation}\label{2.47}
				\I{\Om}\frac{F(u_\e(y))dy}{|x-y|^\mu} \leq |F(u_\e)|_{s_0}\left|\frac{1}{|x-y|^\mu}\right|_{r_0}\leq \widehat{M},
			\end{equation}	
			where $\widehat{M} >0$ is a constant independent of $\e$.\\
			Now, let $ \psi \in H_0^{1}(\Om)$, $\psi \geq 0$. Then since $u_\e$ is a solution of $(\mc P_{\e,\la})$, using \eqref{2.47}, we have
		\begin{align*}
			\I{\Om}\na u_\e \na \psi +\I{\Om}l_\e(u_\e)\psi\leq \la \widehat{M} \I{\Om}f(u_\e)\psi dx. 
		\end{align*}	
	Using \textcolor{red}{$(f_3)$}, \eqref{2.11} and the fact that $q+1<r_0$, we obtain
	\begin{align*}
		\ds \frac{\la \widehat{M} f(t)}{l_\e(t)}=\ds\frac{(t+\e)^{q+\beta}\la \widehat{M}f(t)}{-t^q\log\left(t+\frac{\e}{t+\e}\right)} \leq \ds\frac{(t+1)^{q+\ba}\la \widehat{M}f(t)}{t^{q+1}} \ra 0 ~\text{as} ~ t \ra 0.
	\end{align*}
Hence there exists $0<\de <\de_0$ that again does not depend on $0<\e<\e_0$ such that 
\begin{align*}
\frac{\la \widehat{M} f(t)}{l_\e(t)}\leq \frac{1}{2}~\text{for}~s \leq \de ~\text{and}~0<\e<\e_0.
\end{align*}
Consequently,
\begin{align*}
		\I{\Om}\na u_\e \na \psi \leq \la \widehat{M}\I{\Om \cap \{u_\e >\de\}} f(u_\e)\psi-\I{\Om\cap\{u_\e >1-\e\}}l_\e(u_\e)\psi dx.
\end{align*}
			Using \eqref{2.6} and $(f_2)$, we get
			\begin{align}\nonumber \label{e2.49}
				\I{\Om}\na u_\e\na \psi  \leq & \frac{\la \widetilde{M}}{\de}\I{\{u_\e >\delta \}}u_\e\exp(\al u_\e^{2})\psi dx +\I{\Om}u_\e \psi dx\\
				\leq &C \I{\Om}u_\e\exp(\al u_\e^{2})\psi  dx,~\text{for every}~ \psi \in H_0^{1}(\Om),~\psi >0
			\end{align}		
			where $C>0$ is independent of $\e$.	Rest of the proof follows exactly as the proof of Proposition $1$ in \cite{FMS}. This completes the proof. \QED
			\end{proof}
		Given a sequence $\{\e_k\}$ in the interval $(0, \e_0)$, we denote by $u_{\e_k}$
		the solutions of $(\mc P_{\e,\la})$
		provided by Proposition \ref{p2.9}.
		
		\begin{Lemma}\label{l2.11}
			Suppose that $f$ satisfies $(f_2)$-$(f_5)$. Let $\{\e_k\}$ in $(0,\e_0)$ be a sequence such that $\e_k \ra 0$ as $k \ra \infty$. Then $u_{\e_k}$ has a subsequence which converges weakly in $H_0^{1}(\Om)$ to a nonnegative and nontrivial function $u$.
		\end{Lemma}		
		\begin{proof}
			We know from Proposition \ref{p2.9} that there exists a constant $M>0$ such that 
			\begin{equation*}
				\|u_{\e_K}\| \leq M.
			\end{equation*}
			Then, we may find a subsequence, still denoted by $\{u_{\e_k}\}$  and an element $u \in H_0^{1}(\Om)$ such that
			\begin{align}\label{e2.51}
				\begin{cases}
					u_{\e_k} \rightharpoonup u~\text{weakly in}~H_0^{1}(\Om),\\
					u_{\e_k} \ra u~\text{in}~L^p(\Om)~\text{for every}~p\geq 1,\\
					u_{\e_k}(x)\ra u(x)~\text{a.e. in }~\Om,\\
					|u_{\e_k}|\leq z_p~\text{a.e. in }~\Om~\text{for some}~z_p \in L^p(\Om).
				\end{cases}	
			\end{align}
			Since $u_{\e_k}$ is a critical point of $J_{\e_k,\la}$, we have
			\begin{equation*}
				\|u_{\e_k}\|^2+\I{\Om}l_{\e_K}(u_{\e_k})u_{\e_k}=\la \I{\Om}\I{\Om}\frac{F(u_{\e_k}(y))f(u_{\e_k}(x))u_{\e_k}(x)}{|x-y|^\mu}dydx
			\end{equation*}
			and from Proposition \ref{p2.9}
			\begin{equation*}
				J_{\e_k}(u_{\e_k})=\frac{\|u_{\e_k}\|^2}{2}+\I{\Om}L_{\e_k}(u_{\e_k})-\frac{\la}{2} \I{\Om}\I{\Om}\frac{F(u_{\e_k}(y))F(u_{\e_k}(x))}{|x-y|^\mu}dydx>m_1.
			\end{equation*}
			Hence
			\begin{align}\nonumber \label{eZ}
			m_1<&\I{\Om}\left(L_{\e_k}(u_{\e_k})-\frac{1}{2}l_{\e_k}(u_{\e_k})u_{\e_k}\right) +\la \I{\Om}\left( \I{\Om}\frac{F(u_{\e_k})}{|x-y|^\mu}dy\right)\left(\frac{1}{2}f(u_{\e_k})u_{\e_k}-F(u_{\e_k})\right)dx\\
			=&	J_{\e_k,\la}(u_{\e_k}).
		\end{align}
			Since $|u_{\e_k}|_{\infty}<M_1$ for every $k$, we have 
			\begin{equation}\label{e2.52}
				\left| \I{\Om}\frac{F(u_{\e_k}(y))}{|x-y|^\mu}dy\right| \leq C,~\text{uniformly in}~k~\text{and}~x \in \Om.
			\end{equation}
			Since $f$ satisfies $(f_3)$, we may apply \eqref{e2.46}, \eqref{e2.51}, \eqref{e2.52} and Dominated convergence theorem to obtain
			\begin{equation*}
				\lim\limits_{k \ra \infty}\I{\Om}\I{\Om}\frac{F(u_{\e_k}(y))f(u_{\e_k}(x))u_{\e_k}(x)}{|x-y|^\mu}dydx=\I{\Om}\I{\Om}\frac{F(u(y))f(u(x))u(x)}{|x-y|^\mu}dydx
			\end{equation*}
			and 
			\begin{equation*}
				\lim\limits_{k \ra \infty}\I{\Om}\I{\Om}\frac{F(u_{\e_k}(y))F(u_{\e_k}(x))}{|x-y|^\mu}dydx=\I{\Om}\I{\Om}\frac{F(u(y))F(u(x))}{|x-y|^\mu}dydx.
			\end{equation*}
			Using \eqref{2.10} and \eqref{e2.51}, we have
			\begin{equation*}
				|l_{\e_k}(u_{\e_k})u_{\e_k}|\leq C(1+z_2^{2-\ba})\in L^1(\Om).
			\end{equation*}
		Thus by above estimate and Dominated Convergence theorem, we have
		\begin{equation}
			l_{\e_k}(u_{\e_k})u_{\e_k} \ra u^{1-\ba}\log(u)\chi_{\{u>0\}}~\text{in}~L^1(\Om).
		\end{equation}
		Now, note that 
		\begin{equation*}
			|l_\e(t)|=l_\e(t)\leq -(t+\e)^{-\ba}\log(t+\e)~\text{if}~t\leq 1-\e
		\end{equation*}
		and 
		\begin{equation*}
			|l_{\e}(t)| \leq \left|\log\left(t + \frac{\e}{t+\e}\right)\right|\leq C_1\left(t + \frac{\e}{t+\e}\right)\leq C_1 (t+\e) ~\text{if}~t \geq 1-\e
		\end{equation*}
		where $C_1>0$ is a constant.
		Thus
		\begin{equation*}
			|l_\e(t)|\leq |\log(t+\e)|(t+\e)^{-\ba}+C_1(t+\e):=N(t),~\text{for}~ t \geq 0.
		\end{equation*}
		Therefore, by \eqref{e2.51} we get 
		\begin{equation*}
			|L_{\e_k}(u_{\e_k})|\leq \I{0}^{u_{\e_k}}N(t)dt \leq \I{0}^{z_1}N(t)dt \in L^1(\Om).
		\end{equation*}
		Consequently, by dominated convergence theorem, we have 
		\begin{equation*}
			\I{\Om}L_{\e_k}(u_{\e_k})\ra \I{\Om}L(u),~\text{where}~L(t)=-\I{0}^{t}s^{-\ba}\log(s)ds.
		\end{equation*}
		Taking the above claims into account and letting $k \ra \infty$ in \eqref{eZ}, we conclude that $u$ is nontrivial. \QED.
		\end{proof}		
		\section{Gradient estimates for solution of the approximate problem}\label{GE}		
		In this section, in order to show that $u$ is a weak solution to $(P_\lambda)$, we
		establish a key-point gradient estimate of the solutions $u_\e$ of problem $(\mc P_{\e,\la})$ stated in Proposition \ref{p2.9}. 
		\begin{Lemma}\label{l3.1}
			Suppose $f$ satisfies $(f_1)-(f_6)$. Let $\e_0$ be given by  \eqref{2.11} and define $\e^\ast=\min\{\e_0,1/2\}$, where $\e_0$ is given by \eqref{2.11}. Let $\psi$ be such that
			\begin{equation*}
				\psi \in C^2(\overline{\Om}),~\psi>0~\text{in}~\Om,~\psi =0~\text{on}~\partial{\Om}~\text{and}~\frac{|\na \psi|^2}{\psi}~\text{is bounded in }~\Om.
			\end{equation*}
			For each $0<\e<\e^\ast$, let $u_\e$ be the nonnegative solution of $(\mc P_{\e,\la})$ obtained in Proposition \ref{p2.9}. Then there is constant $K>0$ such that
			\begin{equation}\label{e3.1}
				\psi(x)|\na u_\e|^2\leq KZ(u_\e(x))~\text{for every}~x \in \Om,
			\end{equation}
			where 
			\begin{align*}
				Z(t)=\begin{cases}
				\ds\frac{t^2}{2}+\ds\frac{t^{1-\ba}}{(1-\ba)^2}-\ds\frac{t^{1-\ba}\log t}{1-\ba},&~0\leq t\leq 1, 	\\
				t-\frac{1}{2}+\ds \frac{1}{(1-\ba)^2},&~t \geq 1.
				\end{cases}	
			\end{align*} 
		Also the constant $K$ is independent of $0<\e<\e^\ast$ and depends only on $\Om$, $\psi$ and $M_1$ given by Lemma \ref{l2.10}. 
		\end{Lemma}	
		
		\begin{proof} Using Lemma \ref{l2.10}, we see that the solutions $u_\e$ of $(\mc P_{\e,\la})$ are bounded in $L^{\infty}(\Om)$ by a constant $M_1>0$ independent of $\e$. This implies that right hand side of $(\mc P_{\e.\la})$ is bounded in $L^{\infty}(\Om)$. The standard elliptic regularity theory implies $u_\e \in C^{1,\nu}(\overline{\Om})$ for each $0<\nu<1$.
			Following the approach of Lemma 5.1 in \cite	{FMS}, we define 
			\begin{equation*}
				\overline{h}_\e(x,u)=-\frac{u^q}{(u+\e)^{q+\ba }}\log\left(\frac{u^2+\e u +\e}{u+\e}\right)-\la K_{\e}(x)f(u(x)),
			\end{equation*}
			where \begin{equation*}
				K_\e(x)=\I{\Om}\frac{F(u_\e(y))}{|x-y|^\mu}dy.
			\end{equation*}
			As observed in \eqref{2.47}, we know tha $|K_\e|<C$ for some positive constant $C$ independent of $\e$. 	Now define $Z_a(t)=Z(t)+a$, for some $0<a<1$. Then, we have
			\begin{equation*}
				Z_a\in C^2(0,\infty),~Z_a(t)> 0,~Z_a^\prime(t)> 0~\text{and}~Z_a^{\prime\prime}(t)\leq 0~\text{for all}~t>0.
			\end{equation*}
				Now setting $w=\displaystyle\frac{|\nabla u|^2}{Z_a(u)}$ and $v=w\psi$, we argue by contradiction and assume that \begin{equation}
						\sup_{\Omega}v>K,
						\end{equation}
				where $K>0$ will be chosen later independent of $a$ and $0<\e<\e^\ast$.\\
				First we show that $v \in C^2$ at all points $x \in \Om$ such that $u(x)>0$. Let $x \in \Om$ be one such point. Continuity of $u$ implies the existence of an open ball $U \subset \Om$ centered at $x$ such that $u>0$ in $\overline{U}$. Hence, $l_\e(u) \in C^2(U)$ and $f(u) \in C^{1,\ba}(\Om)$. This alongwith boundedness of $K_\e$ and $K_\e^{\prime}$ give $\overline{h}_\e(u) \in C^{1,\ba}(U)$. Since $u$ satisfies the equation $-\De u +	\overline{h}_\e(u)=0$ in $U$, we conclude that $u \in C^3(U)$, implying that $Z_a(u)$ and $w$ are $C^2$ in $U$. Denoting $x_0=\sup_\Omega v$,
					 we have $v(x_0)>K$. Then $x_0 \in \Om$ since $v=0$ on $\partial \Om$ and so $\nabla v(x_0) =0$ and $\Delta v(x_0)\leq 0$. We will compute $\De v$ and evaluate it at the point $x_0$. As we shall see this leads to the absurd $\De v(x_0)>0$ if one fixes $K$ large enough. Now, the estimate \cite[(55)]{FMS} (see also \cite[p. 100-103]{LM} for more details on computations), in our case, has the form
			\begin{align}\nonumber\label{4.3}
				Z_a(u)	\De v \geq& \psi w^2\left(\frac{1}{2}Z_a^{\prime}(u)^z-Z_a(u)Z_a^{\prime\prime}(u)\right)\\
				&+w\left(2\psi Z_a(u)C_\e(u)-\psi \overline{h}_\e(x,u)Z_a^{\prime}(u)-KZ_a(u)\right)\\ \nonumber
				&-2\psi\la\partial_1 K_\e f(u)Z_a^\frac{1}{2}w^\frac{1}{2}-K^{\prime}Z_a^\prime(u)Z_a(u)^\frac{1}{2}\psi^\frac{1}{2}w^\frac{3}{2},
			\end{align}
			where $C_\e(x,t)=l_\e^\prime(t)- \la K_\e(x) f^\prime(t)$ and $K^{\prime}$ is a positive constant. Now we claim that the following estimates holds uniformly for every $0<\e<\e^\ast$,
			\begin{equation}\label{3.2}
				Z_a^\prime(u)Z_a(u)^\frac{1}{2}\leq C\left(\frac{1}{2}Z_a^{\prime}(u)-Z_a^{\prime\prime}(u)Z_a(u)\right),
			\end{equation}
			\begin{equation}\label{3.3}
				Z_a(u)|C_\e(x,u)| \leq 	C\left(\frac{1}{2}Z_a^{\prime}(u)-Z_a^{\prime\prime}(u)Z_a(u)\right),
			\end{equation}
			\begin{equation}\label{3.4}
				Z_a^\prime(u)\overline{h}_\e(x,u)\leq C\left(\frac{1}{2}Z_a^{\prime}(u)-Z_a^{\prime\prime}(u)Z_a(u)\right),
			\end{equation}
			\begin{equation}\label{3.5}
				Z_a(u)\leq C\left(\frac{1}{2}Z_a^{\prime}(u)-Z_a^{\prime\prime}(u)Z_a(u)\right),
			\end{equation}
			\begin{equation}\label{3.6}
				Z_a(u)^\frac{1}{2}\leq C\left(\frac{1}{2}Z_a^{\prime}(u)-Z_a^{\prime\prime}(u)Z_a(u)\right),
			\end{equation}
			for every $0 \leq u\leq M_1$. Recall that the constant $M_1$ is independent of $\e$. And the constant $C$ depends only on $M_1$ and $\la$ but not on $\e$ and $a$.\\
			Suppose for the time being that \eqref{3.2}-\eqref{3.6} holds. Then \eqref{4.3} implies that
			\begin{align*}
				\De v \geq& \frac{\frac{1}{2} Z_a^{\prime}(u)^2-Z_a^{\prime\prime}(u)Z_a(u)}{Z_a(u)}\left(\psi w^2 -C(w+\psi w^{1/2}+\psi^{1/2}w^{3/2})\right)\\
				\geq & \frac{\frac{1}{2} Z_a^{\prime}(u)^2-Z_a^{\prime\prime}(u)Z_a(u)}{Z_a(u)\psi}\left(v^2-C(v+\psi^{3/2}v^{1/2}+v^{3/2})\right).
			\end{align*}
		Thus if $v(x_0)=\sup v>K$ for some large enough $K$ independent of $a$ and $0<\e <\e^\ast$ we obtain a contradiction to the fact that $\De v(x_0)\leq 0$. Hence there must exist $K>0$ independent of $a$ such that $ \psi |\na u|^2 \leq K Z_a(u)$ in $\Om$. The result then follows by letting $a \ra 0$.\\
		 We now prove the relations \eqref{3.2} to \eqref{3.6}.
	We consider the following two cases:\\
	\textbf{Case 1:} $u\geq 1$. In this case, left hand sides of \eqref{3.2}-\eqref{3.6} are uniformly bounded in the interval $[1,M_1]$. Also note that
	\begin{align*}
		\frac{1}{2}Z_a^{\prime}(u)-Z_a^{\prime\prime}(u)Z_a(u)=\frac{1}{2},~\text{for}~u \geq 1,
	\end{align*}
	and so right hand sides of \eqref{3.2}-\eqref{3.6} are uniformly bounded. This proves \eqref{3.2}-\eqref{3.6}.\\
	\textbf{Case 2:} $0 \leq u\leq 1$. In this case, we have
		\begin{equation*}
			Z_a(u)=	\frac{u^2}{2}+\ds\frac{u^{1-\ba}}{(1-\ba)^2}-\ds\frac{u^{1-\ba}\log u}{1-\ba}+a,~Z_a^{\prime}(u)=u-u^{-\ba}\log (u),~Z_a^{\prime\prime}(u)=1+ \frac{\ba \log(u)}{u^{\ba+1}}-\frac{1}{u^{\ba+1}}.
		\end{equation*}
	We conclude that as $ u \ra 0^+$
	\begin{eqnarray}\label{4.9}
		\displaystyle\ \frac{1}{2} Z'_a(u)-Z_a''(u)Z_a(u)\sim -a\beta u^{-(\beta+1)}\log(u)
	\end{eqnarray}
Clearly, \begin{equation*}
	Z_a(u)=\frac{u^2}{2} +\ds\frac{u^{1-\ba}}{(1-\ba)^2}-\frac{u^{1-\ba}\log(u)}{(1-\ba)}+a \leq C~\text{ for} ~0<u<M_1,
\end{equation*}
and hence assertion \eqref{3.5} holds. Assertion \eqref{3.6} holds similarly. Now using $0\leq u\leq 1$ and boundedness of $Z_a(u)$, we have
\begin{equation}\label{3.8}
	Z_a^{\prime}(u)Z_a(u)^\frac{1}{2}\leq C(1-\log(u)u^{-\ba})\leq  C(1-\log(u)u^{-2\ba})
\end{equation}
On comparing \eqref{4.9} and \eqref{3.8}, we note that \eqref{3.2} holds. Next we prove \eqref{3.4}. For this, we easily get that if $u \leq 1$
\begin{equation*}
	|Z'_a(u)\bar{h_\epsilon}(x, u)]\leq C\frac{\log(u)^2}{u^{2\beta}}
\end{equation*}
from which together with \eqref{4.9} \eqref{3.4} follows.
 Lastly, we prove \eqref{3.3}. Note that \eqref{3.3} holds easily if $1/2 \leq u \leq M_1$. Let us assume that $u <1/2$. Then since $\e <1/2$, we have $u <1-\e$. Now using $(f_6)$, we have
\begin{align*}
	|C_\e(x,u)| \leq&\ds \frac{-u^{q-1}\log\left(u +\frac{\e}{u+\e}\right)}{(u+\e)^{q+\ba+1}}-\ds\frac{\ba u^q\log\left(u +\ds\frac{\e}{u+\e}\right)}{(u+\e)^{q+\ba+1}}+\frac{u^q}{(u+\e)^{q+\ba}}\frac{(u+\e)^2-\e}{(u^2+\e u +\e)(u+\e)}\\
	&+\la|K_\e(x)||f^{\prime}(u)|\\
	\leq &\frac{-u^{q-1}\log(u)}{(u+\e)^{q+\ba+1}}-\frac{\ba u^q\log(u)}{(u+\e)^{q+\ba+1}}+u^{-(\ba+1)}+C.
\end{align*}
Since
\begin{align*}
	\frac{q\e u^{q-1}}{(u+\e)^{q+\ba+1}}\leq q\frac{u^q}{u(u+\e)^q}\frac{\e}{(u+\e)^{\ba+1}}\leq \frac{q}{u^{\ba+1}},
\end{align*}
we obtain
\begin{align*}
	|C_\e(x,u)|\leq C\left(|\log(u)|u^{-(\ba+1)}+u^{-(\ba+1)}+1\right).
\end{align*}
Hence \begin{align}\label{3.10}
	Z_a(u)|C_\e(x,u)|\leq  C\left(|\log(u)|u^{-(\ba+1)}+u^{-(\ba+1)}+1\right).
	\end{align}
	
Comparing \eqref{4.9} and \eqref{3.10}, we conclude that \eqref{3.3} holds. This completes the proof.
		 \QED
		\end{proof}		
		\section{The limit of approximate solutions} \label{las} In this section we prove Theorem \ref{tmr}. We first establish a useful lemma:
		\begin{Lemma}\label{l4.1}
			The function $u^{-\ba}\log(u)\chi_{\Om_+}$ belongs to $L^1_{loc}(\Om)$, where $\Om_+ =\{x \in \Om:u(x)>0\}$ and $u$ is given by Lemma \ref{l2.11}
		\end{Lemma}	
		\begin{proof}
			The proof is similar to the Lemma $6.1$ in \cite{FMS}.\QED
		\end{proof}	
		Now we will prove Theorem \ref{tmr}.\\
		\begin{proof}
		The proof is essentially the same of \cite{LM}. For the sake of completeness we give the complete proof.	We will show that the solutions $u_\e$ of $(\mc P_{\e,\la})$ obtained in Proposition \ref{p2.9} converge to a weak solution, $u$ of $(\mc P_\e)$ as $\e \ra 0^+$. 
		The nontriviality of $u$ is guaranteed by Lemma \ref{l2.11}. Let $\e>0$ small enough 
		and let $\varphi \in C_c^1(\Om)$. We have
		\begin{equation}\label{e4.2}
			\I{\Om}\na u_\e \na \varphi=\I{\Om}-l_\e(u_\e)\varphi +\la \I{\Om}\I{\Om}\frac{F(u_\e(y))f(u_\e(x))\varphi(x)}{|x-y|^\mu}dydx.
		\end{equation}
		Let the cut-off function $\eta \in C^\infty(\mb R)$ such that $0\leq \eta \leq 1$, $\eta(t)=0$ for $t \leq 1$ and $\eta(t)=1$ for $t \geq 2$. For $m>0$ the function $\vartheta:=\varphi\eta(u_\e/m)$ belongs to $C_c^1(\Om)$. By \eqref{e2.46} and Lemma \ref{l3.1}, we have $|\na u_\e|$ is locally bounded independently of $0<\e<\e^\ast$. It then follows from the Arzela-Ascoli Theorem that $u_\e \ra u$ in $C_{loc}^0(\Om)$, and the set $\Om_+=\{x \in \Om:u(x)>0\}$ is open. Let $\widetilde{\Om}$ be an open set such that $support(\varphi)\subset \widetilde\Om$ and $\widetilde{\Om}\subset \Om$. Let $\Om_0=\Om_+ \cap \widetilde{\Om}$. For every $m>0$ there is an $\e_1>0$ such that
		\begin{equation}\label{e4.3}
			u_\e(x)\leq m~\text{for every}~x \in \widetilde{\Om}\setminus \Om_0~\text{and}~0<\e \leq \e_1.
		\end{equation}
		Replacing $\varphi$ by $\vartheta$ in \eqref{e4.2}, we obtain 
		\begin{equation}\label{eq4.4}
			\I{\Om}\na u_\e\na(\varphi \eta(u_\e/m))=\I{\widetilde{\Om}}-l_\e(u_\e)\varphi \eta(u_\e/m)+\la \I{\widetilde{\Om}}\I{\Om} \frac{F(u_\e(y))f(u_\e(x))}{|x-y|^\mu}\varphi \eta(u_\e/m).
		\end{equation}
		We break the right hand side integrals as
		\begin{equation*}
			A_\e:=\I{\Om_0}\left(-l_\e(u_\e)+\la \left(\I{\Om}\frac{F(u_\e(y)}{|x-y|^\mu}dy\right) f(u_\e(x))\right)\varphi\eta(u_\e/m)
		\end{equation*}
		and \begin{equation*}
			B_\e:=\I{\widetilde{\Om}\setminus\Om_0}\left(-l_\e(u_\e)+\la \left(\I{\Om}\frac{F(u_\e(y)}{|x-y|^\mu}dy\right) f(u_\e(x))\right)\varphi\eta(u_\e/m).
		\end{equation*}
		Clearly, $B_\e=0$, whenever $0 <\e<\e_1$ by \eqref{e4.3} and the definition of $\eta$. We claim that
		\begin{equation}\label{e4.5}
			A_\e \ra \I{\Om_0}\left(u^{-\ba}\log(u)+\la \left(\I{\Om}\frac{F(u(y)}{|x-y|^\mu}dy\right) f(u(x))\right)\varphi\eta(u/m)~\text{as} ~\e \ra 0.
		\end{equation}
		In fact, $u_\e \ra u$ uniformly in $\Om_0$. Then it follows from Lemma \ref{l2.10} and the Dominated Convergence Theorem that
		\begin{equation*}
			\la \I{\Om_0}\left(\I{\Om}\frac{F(u_\e(y))}{|x-y|^\mu}dy\right) f(u_\e(x))\varphi\eta(u_\e/m)\ra \la \I{\Om_0}\left(\I{\Om}\frac{F(u(y))}{|x-y|^\mu}dy\right) f(u(x))\varphi\eta(u/m)~\text{as}~\e \ra 0.
		\end{equation*}
		If $u \leq m/2$ then, for $\e>0$ sufficiently small, we have $u_\e<m$. As a consequence, the integral $A_\e$ restricted to this set is zero. If $u >m/2$, then $u_\e >m/4$ for $\e>0$ small enough. We then apply the Dominated Convergence Theorem as $\e \ra 0$ to get \eqref{e4.5}.
		We take the limit in $m$ to conclude that 
		\begin{align}\nonumber
			\I{\Om_0}\left(u^{-\ba}\log(u)+\la \left(\I{\Om}\frac{F(u(y))}{|x-y|^\mu}dy\right) f(u(x))\right)\varphi\eta(u/m)\\
			\ra \I{\Om_0}\left(u^{-\ba}\log(u)+\la \left(\I{\Om}\frac{F(u(y)}{|x-y|^\mu}dy\right) f(u(x))\right)\varphi,
		\end{align}
		as $m \ra 0$ since $\eta(u/m)\leq 1$ and $ u^{-\ba}\log(u)\chi_{\Om_+}+\la\left(\ds\I{\Om}\frac{F(u(y))}{|x-y|^\mu}dy\right) f(u(x))\in L^1(\widetilde{\Om})$ according to Lemma \ref{l4.1}.
		We now proceed with the integral on the left hand side of \eqref{eq4.4},
		\begin{equation}
			\I{\Om}\na u_\e\na (\varphi\eta)(u_\e/m):=\I{\widetilde{\Om}}(\na u_\e\na \varphi)\eta(u_\e/m)+D_\e.
		\end{equation}
		Note that
		\begin{equation*}
			\I{\widetilde{\Om}}(\na u_\e\na \varphi)\eta(u_\e/m) \ra \I{\widetilde{\Om}}(\na u\na \varphi)\eta(u/m)~\text{as}~\e \ra 0,
		\end{equation*}
		since $u_\e \rightharpoonup u$ in $H_0^{1}(\Om)$ and $u_\e \ra u$ uniformly in $\widetilde{\Om}$. Again, by the Dominated Convergence Theorem,
		\begin{equation}
			\I{\widetilde{\Om}}(\na u\na \varphi)\eta(u/m) \ra \I{\widetilde{\Om}}\na u\na \varphi~\text{as}~m \ra 0.
		\end{equation}
		We claim that
		\begin{equation}\label{e4.9}
			D_\e:=\I{\Om}\frac{|\na u_\e|^2}{m}\eta^\prime(u_\e/m)\varphi \ra 0~\text{as}~\e \ra 0~~~(\text{and then as}~m \ra 0).
		\end{equation}
		The estimate $|\na u_\e|^2 \leq K Z(u_\e)$ in $\widetilde{\Om}$ provided by Lemma \ref{l3.1} and recalling that $\eta^\prime(u_\e/m)=0$ if $u_\e \geq 2m$ yield 
			\begin{align*}
				\limsup\limits_{\e \ra 0} |D_\e| \leq& \frac{C}{m}\limsup\limits_{\e \ra 0} \I{\widetilde{\Om}\cap \{m\leq u_\e\leq 2m\}} Z(u_\e)|\eta^\prime(u_\e/m)\varphi|\\
				\leq&2 C\limsup\limits_{\e \ra 0} \I{\widetilde{\Om}\cap \{m\leq u_\e\leq 2m\}}\frac{ Z(u_\e)|\eta^\prime(u_\e/m)\varphi}{u_\e}\\
				\leq&C \sup|\eta^\prime| \sup |\varphi|\lim\limits_{\e \ra 0} \I{\widetilde{\Om}\cap \{m\leq u_\e\leq 2m\}} \frac{Z(u_\e)}{u_\e}\\
				\leq & 2 C l \sup |\varphi| \I{\widetilde{\Om}\cap \{m\leq u\leq 2m\}} \left(c_1|u|+c_2|u|^{-\ba}+c_3|u|^{-\ba}|\log (u)|\right)\chi_{\{ u>0\}},
		\end{align*}
		for every $m>0$. Letting $m \ra 0$ and using Lemma \ref{l4.1}, \eqref{e4.9} is proved. As a immediate consequence of \eqref{eq4.4}-\eqref{e4.9}, we have
		\begin{equation*}
			\I{\Om}\na u\na \varphi=\I{\Om \cap \{u>0\}}\log (u) \varphi +\I{\Om}\I{\Om}\frac{F(u(y))f(u(x))\varphi(x)}{|x-y|^\mu},~\text{for any}~\varphi \in C_c^1(\Om).
		\end{equation*}
		This concludes the proof of Theorem \ref{tmr}. \QED	
		\end{proof}	
			
		\begin{Remark}
			The case of singularity $t^{\ba}\log(t)$ replaced by singularity $|\log t|^{k-2}\log t,~k \in \mb N \setminus \{1\} $ can be tackled in the similar manner. The function $l_\e$ in this case takes the form
			$$l_\e(t)=\begin{cases}
				\left| \log\left(t+\ds \frac{\e}{t+\e}\right)\right|^{k-2}\log\left(t+ \ds\frac{\e}{t+\e}\right)~&\text{if}~t>0,\\
				0~&\text{if}~t=0.
			\end{cases}$$
	Also the $Z-$function to get the gradient estimate in this can be given as
		\begin{align*}	Z(t)=\begin{cases}
			t^2-\ds\I{0}^{t}|\log s|^{k-2}\log s, ~\text{for}~0\leq t \leq t^\ast,\\
			(t^\ast)^2+(t-t^\ast)(2t^\ast-|\log t^\ast|^{k-2}\log t^\ast) -\ds\I{0}^{t^\ast}|\log s|^{k-2}\log s ds,~\text{for}~t \geq t^\ast,
		\end{cases}
	\end{align*}
	where $t^\ast$ is such that
	\begin{equation*}
		\frac{|\log t^\ast|^{k-2}}{t^\ast}=\frac{2}{k-1}.
	\end{equation*}	
			\end{Remark}
		\textbf{Acknowledgement:} The first author thanks the CSIR(India) for financial support in the form of a Senior Research Fellowship, Grant Number $09/086(1406)/2019$-EMR-I. The second author is funded by IFCAM (Indo-French Centre for Applied Mathematics) IRL CNRS 3494.

	\end{document}